\documentclass[12pt,reqno]{amsart}
\usepackage[utf8]{inputenc}
\usepackage{amsmath,amsxtra,latexsym,amscd,amssymb,exscale,mathrsfs,mathtools,xparse}
\usepackage[top=3cm, bottom=3.5cm]{geometry}
\usepackage{geometry}
\usepackage{boxedminipage}
\usepackage[bookmarks=true]{hyperref}
\usepackage{bookmark}
\newtheorem{theorem}{Theorem}[section]
\newtheorem{lemma}[theorem]{Lemma}
\newtheorem{proposition}[theorem]{Proposition}
\newtheorem{corollary}[theorem]{Corollary}

\theoremstyle{definition}
 
\theoremstyle{remark}

\theoremstyle{remark}

\numberwithin{equation}{section}

\makeatletter
\def\subsection{\@startsection{subsection}{2}%
	\z@{.5\linespacing\@plus.7\linespacing}
	{.5\baselineskip}%
	{\normalfont\centering\scshape}%
}
\makeatother

\let\div\undefined
\DeclareMathOperator{\div}{div}
\let\Div\undefined
\DeclareMathOperator{\Div}{Div}

\let\vol\undefined
\DeclareMathOperator{\vol}{vol}

\let\cof\undefined
\DeclareMathOperator{\cof}{cof}
\DeclareMathOperator{\sff}{\mathrm{I\!I}}
\let\tr\undefined
\DeclareMathOperator{\tr}{tr}

\title[ABP method]{Some geometric inequalities by the ABP method}
\author{Doanh Pham}
\address{Department of Mathematics, Rutgers University, Piscataway, NJ 08854, USA}
\email{doanh.pham@rutgers.edu}

\calclayout

\begin{document}
\maketitle

\begin{abstract}
In this paper, we apply the so-called Alexandrov-Bakelman-Pucci (ABP) method to establish some geometric inequalities. We first prove a logarithmic Sobolev inequality for closed $n$-dimensional minimal submanifolds $\Sigma$ of $\mathbb S^{n+m}$. As a consequence, it recovers the classical result that $|\mathbb S^n| \leq |\Sigma|$ for $m = 1,2$. Next, we prove a Sobolev-type inequality for positive symmetric two-tensors on smooth domains in $\mathbb R^n$ which was established by D. Serre when the domain is convex. In the last application of the ABP method, we formulate and prove an inequality related to quermassintegrals of closed hypersurfaces of the Euclidean space.
\end{abstract}

\section{Introduction}

\subsection{A logarithmic Sobolev inequality for minimal submanifolds of the unit sphere}

Logarithmic Sobolev inequalities have a long history of study and become important tools in analysis and geometry, see e.g. \cite{Brendle4, Ecker, Gross, Latala}. In particular, one of them is crucial in Perelman's proof of the monotonicity formula for the entropy under Ricci flow (see \cite{Perelman}). Let $\mathbb N$ be the set of positive integers. The sharp logarithmic Sobolev inequality for submanifolds in the Euclidean spaces, proved by S. Brendle, reads:
\begin{theorem}[\cite{Brendle4}]\label{Brendle logSobolev}
	Let $n,m \in \mathbb N$ and $\Sigma$ be a closed $n$-dimensional submanifold of $\mathbb R^{n+m}$. Then for every positive smooth function $f$ on $\Sigma$, we have
	$$\int_\Sigma f\left(\log f + n + \frac{n}{2}\log(4\pi) - n^2|H|^2\right) - \left(\int_\Sigma f \right) \log \left(\int_\Sigma f \right) \leq \int_\Sigma \frac{|\nabla^\Sigma f|^2}{f},$$
	where $H$ is the mean curvature vector of $\Sigma$.
\end{theorem}

For $n \in \mathbb N$, let $\mathbb B^{n+1} = \{ x \in \mathbb R^{n+1}: |x| < 1\}$ be the unit ball and $\mathbb S^{n} = \partial \mathbb B^{n+1}$ be the unit sphere in $\mathbb R^{n+1}$ respectively. Our goal is to establish a logarithmic Sobolev inequality for closed minimal submanifolds of $\mathbb S^{n+m}$. In the simplest case when the minimal submanifold is $\mathbb S^n$, we recall the following classical fact:

\begin{theorem}[see e.g. \cite{Beckner}]
	For every positive smooth function $f$ on $\mathbb S^n$, we have
	\begin{equation*}
	\int_{\mathbb S^n} f\left(\log f + \log(|\mathbb S^n|) \right) - \left(\int_{\mathbb S^n}  f \right) \log \left(\int_{\mathbb S^n}  f \right) \leq \frac{1}{2n} \int_{\mathbb S^n}  \frac{|\nabla^{\mathbb S^n} f|^2}{f}.
	\end{equation*}
	The equality holds if and only if $f$ is a constant.
\end{theorem}

The first main result in this paper is the following:

\begin{theorem}\label{my logSoblev m = 2}
	Let $n \in \mathbb N$ and $m = 1,2$. Suppose that $\Sigma$ is a closed $n$-dimensional minimal submanifold of $\mathbb S^{n+m}$. Then for every positive smooth function $f$ on $\Sigma$, we have
	\begin{equation}\label{eq:my logSobolev m = 2}
	\int_\Sigma f\left(\log f + \log(|\mathbb S^n|) \right) - \left(\int_\Sigma f \right) \log \left(\int_\Sigma f \right) \leq \frac{n+1}{2n^2} \int_\Sigma \frac{|\nabla^\Sigma f|^2}{f}.
	\end{equation}
The equality holds if and only if $f$ is a constant and $\Sigma$ is a totally geodesic $n$-sphere.
\end{theorem}

Taking $f \equiv 1$ in Theorem \ref{my logSoblev m = 2}, we immediately recover the following classical result:
\begin{corollary}
	Let $n \in \mathbb N$, $m=1,2$, and $\Sigma$ be a closed $n$-dimensional minimal submanifold of $\mathbb S^{n+m}$. Then $|\mathbb S^n| \leq |\Sigma|$. The equality holds if and only if $\Sigma$ is a totally geodesic $n$-sphere.
\end{corollary}
There are several ways to deduce the above Corollary. One way is by using the isoperimetric inequality for minimal surfaces in the Euclidean spaces obtained in \cite{Brendle3}. Other ways, for any codimension, include a monotonicity type argument  (see \cite{Brendle1}) and the method using heat kernel comparisons (see \cite{ChengLiYau}).\\

For an arbitrary codimension $m \geq 2$ of $\Sigma$, we also have:

\begin{theorem}\label{my logSobolev}
Let $n \in \mathbb N$, $m \geq 2$, and $\Sigma$ be a closed $n$-dimensional minimal submanifold of $\mathbb S^{n+m}$. Then for every positive smooth function $f$ on $\Sigma$, we have
\begin{equation}\label{eq:my logSobolev}
\int_\Sigma f\left(\log f + \log\left((n+1)\frac{|\mathbb S^{n+m}|}{|\mathbb S^{m-1}|}\right) \right) - \left(\int_\Sigma f \right) \log \left(\int_\Sigma f \right) \leq \frac{n+1}{2n^2} \int_\Sigma \frac{|\nabla^\Sigma f|^2}{f}.
\end{equation}
\end{theorem}

The inequality (\ref{eq:my logSobolev m = 2}) is a consequence of Theorem \ref{my logSobolev}. Indeed, taking $m = 2$ and using the formula $(n+1) |\mathbb S^{n+2}| = 2\pi|\mathbb S^n|$, the constant in the first integrand of (\ref{eq:my logSobolev}) reduces to $\log(|\mathbb S^n|)$. For $m = 1$, by viewing $\mathbb S^{n+1}$ as the (totally geodesic) equator of $\mathbb S^{n+2}$, we deduce that every minimal hypersurface of $\mathbb S^{n+1}$ is also a minimal submanifold of $\mathbb S^{n+2}$; so (\ref{eq:my logSobolev m = 2}) follows as well in this case.

As the readers will see in details of the proof of Theorem \ref{my logSobolev}, the condition $m \geq 2$ is necessary. More specifically, if one directly takes $m = 1$ in the proof of the Theorem, the last step of the argument does not hold. For this reason, we provide in Section \ref{sec:my logSoblev m = 1} a complete (and slightly simpler) proof of (\ref{eq:my logSobolev m = 2}) in codimension $1$ setting, despite some repetitions and the fact that it and the proof for $m \geq 2$ are similar.\\
\\
\textit{Relation to Yau's conjecture}: Assume that $\Sigma$ is a closed minimal hypersurface of $\mathbb S^{n+1}$. Let $0 < \lambda_1(\Sigma)$ be the first eigenvalue of the Laplacian operator $- \Delta_\Sigma$ on $\Sigma$. By a direct computation, the restriction to $\Sigma$ of any linear functional $L$ on $\mathbb R^{n+1}$ solves
	$$-\Delta_\Sigma L = nL.$$
Thus $\lambda_1(\Sigma) \leq n$. Recall that $\lambda_1(\mathbb S^n) = n$. It has been a conjecture by S.T. Yau that $ \lambda_1(\Sigma) = n$ (see \cite{Yau1}). In \cite{TangYang1}, Z. Tang and W. Yan verified the conjecture when $\Sigma$ is isoparametric. In the general case, Choi and Wang \cite{ChoiWang1} proved that $\lambda_1(\Sigma) \geq n/2$. Later, by a careful analysis of Choi and Wang's argument, S. Brendle \cite{Brendle2} pointed out that the inequality is strict, i.e. $\lambda_1(\Sigma) > n/2$. The relation between Yau's conjecture and the logarithmic Sobolev inequality is based on the following fact:

\begin{theorem}[see e.g. \cite{Latala}]\label{Poincare and logSobolev}
Let $n, m \in \mathbb N$ and $\Sigma$ be an $n$-dimensional submanifold of $\mathbb R^{n+m}$. Suppose there exists a constant $\alpha > 0$ such that: for every positive smooth function $f$ on $\Sigma$, it holds
$$\int_\Sigma f(\log f + \log |\Sigma|) - \left(\int_\Sigma f \right) \log\left(\int_\Sigma f \right) \leq \frac{1}{2\alpha} \int_\Sigma \frac{|\nabla^\Sigma f|^2}{f},$$
then $\lambda_1(\Sigma) \geq \alpha$.
\end{theorem}
In view of Theorem \ref{Poincare and logSobolev}, if one can improve the term $\log(|\mathbb S^n|)$ in Theorem \ref{my logSoblev m = 2} to $\log(|\Sigma|)$, then it would follow that $\lambda_1(\Sigma) \geq n^2/(n+1)$.

The proofs of Theorems \ref{my logSobolev} and \ref{my logSoblev m = 2} are given in Section \ref{Proof of the logarithmic Sobolev inequality}. They are inspired by the works of S. Brendle \cite{Brendle5, Brendle3, Brendle4}

\subsection{A Sobolev-type inequality for positive symmetric tensors}

Let $n \in \mathbb N$ and $A$ be an $n \times n$ matrix-valued function in a bounded domain $\Omega \subset \mathbb R^n$. The divergence of $A$ is the vector field in $\Omega$ defined by
$$\Div A(x) \coloneqq
\renewcommand{\arraystretch}{1.1}
\begin{pmatrix}
\sum_{j = 1}^n \partial_{x_j} A_{1j}(x)\\
\sum_{j = 1}^n \partial_{x_j} A_{2j}(x)\\\
\vdots\\
\sum_{j = 1}^n \partial_{x_j} A_{nj}(x)\
\end{pmatrix}.
$$
Furthermore, when $\Div A = 0$ in $\Omega$, we say that $A$ is \textit{divergence-free}.

In \cite{Serre}, D. Serre proved an inequality involving positive symmetric matrix-valued functions in $\mathbb R^n$ and showed some applications to fluid dynamics. In smooth cases, it reads:

\begin{theorem}[\cite{Serre}, Theorem 2.3 and Proposition 2.2] \label{Serre} Let $n \in \mathbb N$ and $\Omega$ be a smooth bounded convex domain in $\mathbb R^n$. If $A$ is a smooth uniformly positive symmetric matrix-valued function on $\overline{\Omega}$ (the closure of $\Omega$), then
\begin{equation}\label{eq:Serre}
n^{\frac{n-1}{n}}\left| \mathbb S^{n-1} \right|^{\frac{1}{n}} \left(\int_\Omega \det(A(x))^{\frac{1}{n-1}} dx\right)^{\frac{n-1}{n}} \leq \int_{\partial \Omega} |A(x) \nu| d\sigma(x) + \int_\Omega |\Div A(x)| dx,
\end{equation}
where $\nu$ is the unit outer normal vector field on $\partial \Omega$. When $A$ is divergence-free, the equality in (\ref{eq:Serre}) holds if and only if there is a smooth convex function $u$ on $\Omega$ such that $\nabla u(\Omega)$ is a ball centered at the origin and $A = (\cof D^2u)$, the cofactor matrix of $D^2 u$.
\end{theorem}
The original proof of Theorem \ref{Serre} uses an optimal transportation technique. Because of the regularity issue of the transport map, the convexity condition on $\Omega$ is required in the argument (see e.g. \cite{Villainibook}). Our next result in this paper says that Theorem \ref{Serre} still holds when the convexity condition of $\Omega$ is removed. In order to prove it, we use the ABP method with a suitably chosen change-of-variable map instead of the transport map.

\begin{theorem}\label{my improved Serre}
Let $n \in \mathbb N$ and $\Omega$ be a smooth bounded domain in $\mathbb R^n$. If $A$ is a smooth uniformly positive symmetric matrix-valued function on $\overline{\Omega}$, then
\begin{equation}\label{eq:my improved Serre}
n^{\frac{n-1}{n}}\left| \mathbb S^{n-1} \right|^{\frac{1}{n}} \left(\int_\Omega (\det A(x))^{\frac{1}{n-1}} dx\right)^{\frac{n-1}{n}} \leq \int_{\partial \Omega} |A(x) \nu| d\sigma(x) + \int_\Omega |\Div A(x)| dx,
\end{equation}
where $\nu$ is the unit outer normal vector field on $\partial \Omega$. Moreover, the equality in (\ref{eq:my improved Serre}) holds if and only if there is a smooth convex function $u$ on $\Omega$ such that $\nabla u(\Omega)$ is a ball centered at the origin and $A = (\cof D^2u)$, the cofactor matrix of $D^2 u$.
\end{theorem}
Note that in the rigidity statement of Theorem \ref{my improved Serre}, we do not assume that $A$ is divergence-free in advance. Taking $A = I_n$, the identity matrix, in Theorem \ref{my improved Serre}, we recover the classical isoperimetric inequality for smooth domains.

The proof of the Theorem is provided in Section \ref{Proof of my improved Serre}. It is inspired by the works of S. Brendle \cite{Brendle3} and X. Cabré \cite{Cabre1, Cabre2}.

\subsection{An inequality related to quermassintegrals}\label{subsec:An inequality related to quermassintegrals}

Our next application of the ABP method is a geometric inequality related to quermassintegrals. For a number $n \in \mathbb N$ and an $(n \times n)$-symmetric matrix $A$, let $\sigma_k (A)$, $1\leq k \leq n$, be the $k$-elementary symmetric functions of its eigenvalues. More specifically, if $\lambda_1, \lambda_2, \dots, \lambda_n$ are the eigenvalues of $A$, then
$$\sigma_k (A) = \sum_{1 \leq i_1 < i_2 < \dots < i_k \leq n} \lambda_{i_1}\lambda_{i_2}\dots\lambda_{i_k}.$$

The normalization of $\sigma_k$ is denoted by $H_k(A) \coloneqq \sigma_k(A) / {n \choose k}$. For $0 \leq k \leq n-1$, the $k$-Newton tensor of $A$ is the symmetric matrix $T_k(A)$ given by
$$T_k(A)_{ij} \coloneqq \frac{\partial \sigma_{k+1}(A)}{\partial A_{ij}}.$$
 Furthermore, the Gårding cone $\Gamma_k$ is defined as the set of all symmetric matrices $A$ with the property that $H_j(A) > 0$ for all $1 \leq j \leq k$.
If $A \in \Gamma_{k+1}$, then by Gårding's theory for hyperbolic polynomials, $T_k(A)$ is positive definite (see \cite{Garding}).

Let $\Sigma$ be a closed hypersurface in $\mathbb R^{n+1}$, and $\nu$ be the unit normal vector field on $\Sigma$ which points to the enclosed domain. We use $\sff_\nu$ to denote the scalar second fundamental form on $\Sigma$ in the direction of $\nu$. With these notations, $\sff_\nu$ is positive semi-definite if $\Sigma$ is the boundary of a convex domain. For simplicity, we will write $H_k$ and $T_k$ instead of $H_k(\sff_\nu)$ and $T_k(\sff_\nu)$, respectively. The hypersurface is said to be $k$-{\it convex} if $\sff_\nu \in \Gamma_k$ at every $x \in \Sigma$. Our next main result in this section is the following:

\begin{theorem}\label{my quermass 2nd}
	Let $n \in \mathbb N$, $n \geq 2$, and $0 \leq k \leq n-2$. Suppose that $\Sigma$ is a closed $(k+1)$-convex hypersurface in $\mathbb R^{n+1}$. Then, we have
	\begin{equation}\label{eq:my quermass 2nd}
	|\mathbb S^n| \left(\int_\Sigma H_k \right)^{n+1} \leq {n-1 \choose k}^n \left(\int_\Sigma H_{k+1} \right)^{n+1} \int_\Sigma \frac{H_k^{n+1}}{(\det T_k)H_{k+1}}.
	\end{equation}
	The equality holds if and only if $\Sigma$ is a sphere.
\end{theorem}

When $\Sigma$ is $(k+1)$-convex, the determinant of the $k$-Newton tensor of its second fundamental form satisfies $\det T_k \geq {n-1 \choose k}^n H_{k+1}^{nk/(k+1)}$ (see Section \ref{sec:Proof of my quermass} for more details). Hence, as a consequence of Theorem \ref{my quermass 2nd}, we obtain

\begin{corollary}\label{my quermass result}
	Let $n \in \mathbb N$, $n \geq 2$, and $0\leq k \leq n-2$. Suppose that $\Sigma$ is a closed $(k+1)$-convex hypersurface in $\mathbb R^{n+1}$. Then, we have
	\begin{equation}\label{eq:my quermass 1st}
	|\mathbb S^n| \left(\int_\Sigma H_k \right)^{n+1} \leq \left(\int_\Sigma H_{k+1} \right)^{n+1} \left(\int_\Sigma \frac{H_k^{n+1}}{H_{k+1}^{\frac{nk}{k+1}+1}} \right).
	\end{equation}
The equality holds if and only if $\Sigma$ is a sphere.
\end{corollary}

{\it Relation of Corollary \ref{my quermass result} to the Alexandrov-Fenchel-type inequalities.} According to the classical Brunn–Minkowski theory of mixed volumes (see e.g. \cite{Schneiderbook}), when $\Sigma$ is $n$-convex (equivalently, it is the boundary of a convex domain), a consequence of the Alexandrov-Fenchel inequality states that
\begin{equation}\label{eq:A-F inequality}
|\mathbb S^n|\left(\int_\Sigma H_k \right)^{n-k-1} \leq \left(\int_\Sigma  H_{k+1} \right)^{n-k} \quad \text{for } 0 \leq k \leq n-2,
\end{equation}
with the equality holds if only if $\Sigma$ is a sphere. It has been an active research direction to generalize (\ref{eq:A-F inequality}) to hold on hypersurfaces with some relaxed curvature conditions. In \cite{GuanLi1}, by the method of curvature flow, P. Guan and J. Li proved  (\ref{eq:A-F inequality}) for $(k+1)$-convex, star-shaped hypersurfaces (see also \cite{GuanLi2} where the methods of geometric flows are discussed). In \cite{ChangWang1}, by the method of optimal transportation, Alice Chang and Y. Wang proved (\ref{eq:A-F inequality}), with non-sharp constant, for $(k+2)$-convex hypersurfaces. In some end cases, they were able to give the sharp constants (see \cite{ChangWang2}).

In general, $(k+1)$-convexity is a natural condition for (\ref{eq:A-F inequality}). A major open problem in this research direction is the following:

\textbf{Open Problem:} Does the inequality (\ref{eq:A-F inequality}) hold under the assumption that $\Sigma$ is $(k+1)$-convex ?

If the Open Problem holds, then Corollary \ref{my quermass result} is its consequence. To see this, we assume that $0 \leq k \leq n-2$ and $\Sigma$ is $(k+1)$-convex. Then, by Maclaurin's inequality, we have
$$H_k \leq H_k^\frac{n+1}{k+2}/H_{k+1}^\frac{nk-k(k+1)}{(k+1)(k+2)}.$$
Consequently, by Holder's inequality, it follows that
\begin{equation}\label{eq:easy Holder}
\int_\Sigma H_k \leq \int_\Sigma \frac{H_k^\frac{n+1}{k+2}}{H_{k+1}^\frac{nk-k(k+1)}{(k+1)(k+2)}} \leq \left( \int_\Sigma H_{k+1} \right)^\frac{k+1}{k+2} \left( \int_\Sigma \frac{H_k^{n+1}}{H_{k+1}^{\frac{nk}{k+1}+1}} \right)^\frac{1}{k+2}.
\end{equation}
Thus, if we \textit{assume} that the \textit{Open Problem has affirmative answer}, then we can combine it with (\ref{eq:easy Holder}) to deduce (\ref{eq:my quermass 1st}).

The proof of Theorem \ref{my quermass 2nd} and more details on the deduction of Corollary \ref{my quermass result} are given in Section \ref{sec:Proof of my quermass}.

\section{Proof of the logarithmic Sobolev inequality for minimal submanifolds of $\mathbb S^{n+m}$}\label{Proof of the logarithmic Sobolev inequality}

\subsection{Proof of (\ref{eq:my logSobolev m = 2}) for closed minimal hypersurfaces of $\mathbb S^{n+1}$}\label{sec:my logSoblev m = 1}
In this subsection, we give a proof of (\ref{eq:my logSobolev m = 2}) in codimension $1$ setting. More precisely, \textit{for $n \in \mathbb N$, let $\Sigma$ be a closed minimal hypersurface of the unit sphere $\mathbb S^{n+1}$ in $\mathbb R^{n+2}$, and $f$ be a positive smooth function on $\Sigma$. We will prove that}
\begin{equation}\label{eq:my logSobolev m = 1}
\int_\Sigma f\left(\log f + \log(|\mathbb S^n|) \right) - \left(\int_\Sigma f \right) \log \left(\int_\Sigma f \right) \leq \frac{n+1}{2n^2} \int_\Sigma \frac{|\nabla^\Sigma f|^2}{f}.
\end{equation}
To do this, we suppose for now that $\Sigma$ is connected. After multiplying $f$ by a (unique) positive constant, we may assume without loss of generality that
\begin{equation}\label{eq:scaling in my logSobolev m=1}
\frac{n}{n+1} \int_\Sigma f(x) \log(f(x)) \; d\vol_\Sigma(x) = \frac{1}{2n} \int_\Sigma \frac{|\nabla^\Sigma f(x)|^2}{f(x)}\; d\vol_\Sigma(x).
\end{equation}
Then, we need to show
\begin{equation}\label{eq:my logSobolev m=1 goal}
|\mathbb S^n| \leq \int_\Sigma f(x)\; d\vol_\Sigma(x).
\end{equation}
Since $f > 0$ and $\Sigma$ is connected, by (\ref{eq:scaling in my logSobolev m=1}), there exists (uniquely modulo a constant) a smooth solution $u$ to the equation
$$\div(f\,\nabla^\Sigma u) = \frac{n}{n+1} f \log f - \frac{1}{2n} \frac{|\nabla^\Sigma f|^2}{f}.$$
Let $\nu$ be a unit normal vector field of $\Sigma$ in $\mathbb S^{n+1}$. We denote $\sff_\nu$ as the scalar second fundamental of $\Sigma$ in the direction of $\nu$. More precisely, this means
$$\sff_\nu(V,W) = \langle \overline{\nabla}_V W, \nu \rangle = -\langle \overline{\nabla}_V \nu, W \rangle  \quad \text{for every } V, W \in T_x\Sigma,$$
where  $\overline{\nabla}$ is the Euclidean connection on $\mathbb R^{n+2}$. We define the sets $V$ and $\Omega$ by
\begin{align*}
&V \coloneqq \{(x,s,t):\; x \in \Sigma, \, s, t \in \mathbb R\, \text{ such that }\\
& \qquad \qquad \qquad |\nabla^\Sigma u(x)|^2 + s^2 + t^2 < 1 \text{ and } D^2_\Sigma u(x) - s\sff_\nu(x) + \;  tg_\Sigma(x) \geq 0\},\\
&\Omega \coloneqq \{x \in \Sigma \, \text{ such that } (x,s,t) \in V \, \text{ for some } s, t \in \mathbb R\},
\end{align*}
where $g_\Sigma$ denotes the induced metric on $\Sigma$. In particular, since $\Sigma$ is a minimal submanifold of $\mathbb S^{n+1}$, we have
\begin{equation}\label{eq:lower bound of Laplacian u + nt m=1}
0 \leq \tr_{g_\Sigma} (D_\Sigma^2 u(x) - s\sff_\nu(x) + \; tg_\Sigma(x)) = \Delta_\Sigma u(x) + nt \quad \text{for } (x,s,t) \in V.
\end{equation}
We also define a map $\Phi : \Sigma \times \mathbb R \times \mathbb R \longrightarrow \mathbb R^{n+2}$ by
$$\Phi(x,s,t) =\nabla^\Sigma u(x) + s\nu + tx \quad \text{for } x \in \Sigma,\; s, t \in \mathbb R.$$

\begin{lemma}\label{Phi(V) in logSobolev m=1}
	We have $\mathbb B^{n+2} = \Phi(V)$.
\end{lemma}
\begin{proof}
	It is clear that $\Phi(V) \subseteq \mathbb B^{n+2}$. To see the reverse inclusion, let us pick a point $\xi$ in $\mathbb B^{n+2}$. For every $x\in \Sigma$, we denote $\xi^{\text{tan}}_x$ the image of $\xi$ under the orthogonal projection onto $T_x\Sigma$. In other words,
	$$\langle \xi^{\text{tan}}_x, V \rangle = \langle \xi , V \rangle \quad \text{for every } V \in T_x\Sigma.$$
	We consider the function
	$$v(x) \coloneqq u(x) - \langle \xi, x \rangle \quad \text{for } x \in \Sigma.$$
	Since $\Sigma$ is closed, the function $v$ attains its global minimum at some point $x_0 \in \Sigma$. Then
	$$0 = \nabla^\Sigma v(x_0) = \nabla^\Sigma u(x_0) - \xi^{\text{tan}}_{x_0}.$$
	Hence, $\xi = \nabla^\Sigma u(x_0) + s_0 \nu + t_0 x_0$ where $s_0 = \langle \xi, \nu \rangle$ and $t_0 = \langle \xi, x_0 \rangle$. In particular, since $\nabla^\Sigma u(x_0)$, $x_0$, and $\nu$ are perpendicular to each other, it follows that $|\nabla^\Sigma u(x_0)|^2 + s_0^2 + t_0^2 = |\xi|^2 < 1$. Moreover, we also have
	$$0 \leq D_\Sigma^2 v(x_0) = D_\Sigma^2 u(x_0) - s_0\sff_\nu + \;  t_0 g_\Sigma(x_0).$$
	Therefore, $(x_0, s_0, t_0) \in V$ and $\xi = \Phi(x_0, s_0, t_0)$. Since $\xi$ is arbitrarily chosen, this finishes the proof of the Lemma.
\end{proof}

\begin{lemma}\label{det(D Phi) in logSobolev m=1}
	For every $(x,s,t) \in V$, we have
	\begin{equation}\label{eq:bound of t in logSoblev m=1}
f(x)^{\frac{1}{n+1}} - \sqrt{1 - |\nabla^\Sigma u(x)|^2} + t \geq 0.
	\end{equation}
	Moreover, for every $(x,s,t) \in V$, the Jacobian determinant of $\Phi$ satisfies
	\begin{equation}\label{eq:det(D Phi) in logSobolev m=1}
	|\det D\Phi(x,s,t)| \leq \left( f(x)^{\frac{1}{n+1}} - \sqrt{1 - |\nabla^\Sigma u(x)|^2} + t \right)^n.
	\end{equation}
\end{lemma}
\begin{proof}
	To start the proof of this lemma, we fix an arbitrary point $(x,s,t) \in V$. We first show (\ref{eq:bound of t in logSoblev m=1}). By the equation of $u$, we have
	\begin{align*}
	\Delta_\Sigma u(x) &= \frac{n}{n+1} \log f(x) - \frac{1}{2n}\frac{|\nabla^\Sigma f|^2}{f^2} - \left\langle  \frac{\nabla^\Sigma f}{f}, \nabla^\Sigma u \right\rangle.\\
	& = \frac{n}{n+1} \log f(x) + \frac{n}{2} |\nabla^\Sigma u|^2 - \frac{1}{2}\left| \frac{1}{\sqrt{n}}\frac{\nabla^\Sigma f}{f} + \sqrt{n} \nabla^\Sigma u \right|^2\\
	& \leq \frac{n}{n+1} \log f(x) + \frac{n}{2} |\nabla^\Sigma u|^2.
	\end{align*}
	Using the elementary inequalities
	$$\log \lambda \leq \lambda - 1 \text{\; for } \lambda > 0 \quad \text{and} \quad \sqrt{1 - \theta} \leq 1 -\frac{\theta}{2} \text{\; for } 0 \leq \theta \leq 1,$$
	we get
	\begin{equation}\label{eq:my logSobolev elem ineq m=1}
	\frac{1}{n+1} \log f(x) \leq f(x)^{\frac{1}{n+1}} - 1 \quad \text{and} \quad \sqrt{1 - |\nabla^\Sigma u|^2} \leq 1 -\frac{|\nabla^\Sigma u|^2}{2}.
	\end{equation}
	Hence, we continue the chain of estimates for $\Delta_\Sigma u$ by
	\begin{equation}\label{eq:Laplacian u in logSobolev m=1}
	\Delta_\Sigma u(x) \leq n\left( f(x)^{\frac{1}{n+1}} - 1 + \frac{|\nabla^\Sigma u|^2}{2}\right) \leq n\left(f(x)^{\frac{1}{n+1}} - \sqrt{1 - |\nabla^\Sigma u|^2}\right).
	\end{equation}
	Consequently, it follows from (\ref{eq:lower bound of Laplacian u + nt m=1}) and (\ref{eq:Laplacian u in logSobolev}) that
	\begin{equation}\label{eq:bound of t in logSoblev, proof m=1}
	0 \leq \Delta_\Sigma u(x) + nt \leq n\left(f(x)^{\frac{1}{n+1}} - \sqrt{1 - |\nabla^\Sigma u|^2} + t\right).
	\end{equation}
	This verifies (\ref{eq:bound of t in logSoblev m=1}).
	
	To show (\ref{eq:det(D Phi) in logSobolev m=1}), we choose a local orthonomal frame $\{e_1, \dots, e_n\}$ in a neighborhood of $x$ in $\Sigma$. With respect to the frame $\{e_1, \dots, e_n,\nu, x\}$, we compute
	\begin{align*}
	\langle \overline{\nabla}_{e_i} \Phi, e_j \rangle &= \langle \overline{\nabla}_{e_i} \nabla^\Sigma u, e_j \rangle + s\langle \overline{\nabla}_{e_i} \nu, e_j \rangle + t\langle \overline{\nabla}_{e_i} x, e_j \rangle\\
	&= (D^2_\Sigma u)(e_i, e_j) - s\sff_\nu(e_i, e_j) + t\delta_{ij}.
	\end{align*}
	Moreover, since
	$$\frac{\partial}{\partial s} \Phi = \nu \quad \text{and} \quad \frac{\partial}{\partial t} \Phi = x,$$
	we infer that
	$$\left\langle \frac{\partial}{\partial s} \Phi, \, e_j \right\rangle = 0, \quad \left\langle \frac{\partial}{\partial s} \Phi, \,\nu \right\rangle = 1, \quad \left\langle \frac{\partial}{\partial s} \Phi, \, x \right\rangle = 0, $$
	and
	$$\left\langle \frac{\partial}{\partial t} \Phi, \, e_j \right\rangle = 0, \quad \left\langle \frac{\partial}{\partial t} \Phi, \,\nu \right\rangle = 0, \quad \left\langle \frac{\partial}{\partial t} \Phi, \, x \right\rangle = 1.$$
	Therefore, the matrix representing $D\Phi$ has the form
	\begin{equation}\label{eq:matrix of DPhi in logSobolev m=1}
	D\Phi(x,y,t) = \left(\begin{array}{@{}c | c@{}}
	D_\Sigma^2 u(x) - s\sff_\nu + \; tI_n& \quad * \quad \\
	\hline
	0 & I_{2}
	\end{array}\right).
	\end{equation}
	Consequently, since $D_\Sigma^2 u(x) - s\sff_\nu + \; tI_n \geq 0$ as $(x,s,t) \in V$, it follows from (\ref{eq:bound of t in logSoblev, proof m=1}), (\ref{eq:matrix of DPhi in logSobolev m=1}), and the arithmetic-geometric mean inequality that
	\begin{align}
	|\det D\Phi(x,s,t)| &= \det(D_\Sigma^2 u(x) - s\sff_\nu + \;  tI_n) \nonumber \\
	&\leq \left( \frac{\Delta_\Sigma u(x) + nt}{n} \right)^n \nonumber \\
	&\leq \left( f(x)^{\frac{1}{n+1}} - \sqrt{1 - |\nabla^\Sigma u(x)|^2} + t \right)^n.\label{eq:det(DPhi) in logSobolev less than (Laplacian u/n + t)^n m=1}
	\end{align}
	This verifies (\ref{eq:det(D Phi) in logSobolev m=1}) and finishes the proof of the Lemma.
\end{proof}

By (\ref{eq:bound of t in logSoblev m=1}) and the fact that $\Phi(V) = \mathbb B^{n+2}$ from Lemma \ref{Phi(V) in logSobolev m=1}, for any $(x,s,t) \in V$, the range of $t$ is given by
$$\sqrt{1 - |\nabla^\Sigma u(x)|^2} - f(x)^{\frac{1}{n+1}} \leq t \leq \sqrt{1 - |\nabla^\Sigma u(x)|^2}.$$
In the next step, we follow the argument in \cite{Brendle5}. We recall that $|\Phi(x,s,t)|^2 = |\nabla^\Sigma u(x)|^2 + s^2 + t^2$. Combining (\ref{det(D Phi) in logSobolev m=1}), Lemma \ref{Phi(V) in logSobolev m=1}, and Fubini's Theorem, we obtain
\begin{align*}
\pi |\mathbb B^{n+1}| &= \int_{\mathbb B^{n+2}} \frac{1}{\sqrt{1 - |\xi|^2}} d\xi\\
&\leq \int_\Omega \int_{\sqrt{1 - |\nabla^\Sigma u|^2} - f^{\frac{1}{n+1}}}^{\sqrt{1 - |\nabla^\Sigma u|^2}} \int_{-\sqrt{1 - |\nabla^\Sigma u(x)|^2 - t^2}}^{\sqrt{1 - |\nabla^\Sigma u(x)|^2 - t^2}}\\
& \qquad |\det D\Phi(x,s,t)|\frac{1}{\sqrt{1 - |\nabla^\Sigma u(x)|^2 - t^2 - s^2}}\, 1_V(x,s,t)\, ds\, dt\, d\vol_\Sigma(x)\\
&\leq \int_\Omega \int_{\sqrt{1 - |\nabla^\Sigma u|^2} - f^{\frac{1}{n+1}}}^{\sqrt{1 - |\nabla^\Sigma u|^2}} \int_{-\sqrt{1 - |\nabla^\Sigma u(x)|^2 - t^2}}^{\sqrt{1 - |\nabla^\Sigma u(x)|^2 - t^2}}\\
& \quad \left( f(x)^{\frac{1}{n+1}} - \sqrt{1 - |\nabla^\Sigma u|^2} + t \right)^n \frac{1}{\sqrt{1 - |\nabla^\Sigma u(x)|^2 - t^2 - s^2}} ds\, dt\, d\vol_\Sigma(x)\\
&= \pi \int_\Omega \int_{\sqrt{1 - |\nabla^\Sigma u|^2} - f^{\frac{1}{n+1}}}^{\sqrt{1 - |\nabla^\Sigma u|^2}} \left( f(x)^{\frac{1}{n+1}} - \sqrt{1 - |\nabla^\Sigma u|^2} + t \right)^n dt\, d\vol_\Sigma(x)\\
&= \frac{\pi}{n+1} \int_\Omega f(x)\, d\vol_\Sigma(x),
\end{align*}
where we have used the identity  $\int_{-a}^a \frac{1}{\sqrt{a^2 - s^2}}ds = \pi$ in the second last equality of the chain. In view of (\ref{eq:my logSobolev m=1 goal}), the proof of (\ref{eq:my logSobolev m = 2}) for codimension $1$ is complete when $\Sigma$ is connected.

We now suppose that $\Sigma$ is disconnected. Then, we apply (\ref{eq:my logSobolev m = 1}) to each individual of connected components of $\Sigma$, sum over them, and use the elementary inequality
$$a \log a + b\log b < (a+b)\log(a+b) \quad \text{for all } a, b > 0$$
to finish the proof in this case.

\subsection{Proof of Theorem \ref{my logSobolev}}

Let $n, m \in \mathbb N$ and $\Sigma$ be a closed $n$-dimensional submanifold of $\mathbb S^{n+m}$ with the induced metric $g_\Sigma$. We denote $T_x^\perp \Sigma$ the normal space of $\Sigma$ at $x$ with respect to $\mathbb S^{n+m}$ (rather than $\mathbb R^{n+m+1}$). For every $y \in T_x^\perp \Sigma$, we denote $\sff_y$ the scalar second fundamental form of $\Sigma$ at $x$ in the direction of $y$. This means
$$\sff_y(V, W) =  \langle \sff(V, W), y \rangle = \langle \overline{\nabla}_V W, y \rangle \quad \text{for every } V, W \in T_x\Sigma,$$
where  $\overline{\nabla}$ is the Euclidean connection on $\mathbb R^{n+m+1}$ and $\sff(V, W) \coloneqq (\overline{\nabla}_V W)^\perp$ is the second fundamental form on $\Sigma$ with respect to $\mathbb S^{n+m}$. The mean curvature vector $H$ on $\Sigma$ is the average of principle directions of $\sff$. Thus, $H$ is the zero vector field when $\Sigma$ is a minimal submanifold of $\mathbb S^{n+m}$.\\

To prove Theorem \ref{my logSobolev}, we suppose $\Sigma$ is a closed $n$-dimensional minimal submanifold of $\mathbb S^{n+m}$ where $m \geq 2$. Suppose for now that $\Sigma$ is connected. Let $f$ be a positive smooth function on $\Sigma$. After multiplying $f$ by a (unique) positive constant, we may assume without loss of generality that
\begin{equation}\label{eq:scaling in logSobolev}
\frac{n}{n+1} \int_\Sigma f(x) \log(f(x)) \; d\vol_\Sigma(x) = \frac{1}{2n} \int_\Sigma \frac{|\nabla^\Sigma f(x)|^2}{f(x)}\; d\vol_\Sigma(x).
\end{equation}
Then, we need to show
\begin{equation}\label{eq:goal of proof my logSobolev}
(n+1)\frac{|\mathbb S^{n+m}|}{|\mathbb S^{m-1}|} \leq \int_\Sigma f(x)\; d\vol_\Sigma(x).
\end{equation}
Since $f > 0$ and $\Sigma$ is connected, by (\ref{eq:scaling in logSobolev}), there exists (uniquely modulo a constant) a smooth solution $u$ to the equation
$$\div(f\,\nabla^\Sigma u) = \frac{n}{n+1} f \log f - \frac{1}{2n} \frac{|\nabla^\Sigma f|^2}{f}.$$
For each $0 \leq r < 1$, we define
\begin{align*}
&V_r \coloneqq \{(x,y,t):\; x \in \Sigma, \, y \in T_x^\perp \Sigma, \, t \in \mathbb R\, \text{ such that }\\
& \qquad \qquad r^2 \leq  |\nabla^\Sigma u(x)|^2 + |y|^2 + t^2 < 1 \text{ and } D^2_\Sigma u(x) - \sff_y + \;  tg_\Sigma(x) \geq 0\},\\
&\Omega \coloneqq \{x \in \Sigma \, \text{ such that } (x,y,t) \in V_0 \, \text{ for some } y \in T_x^\perp \Sigma \text{ and } t \in \mathbb R\}.
\end{align*}
In particular, since $\Sigma$ is a minimal submanifold of $\mathbb S^{n+m}$, we have
\begin{equation}\label{eq:lower bound of Laplacian u + nt}
	0 \leq \tr_{g_\Sigma} (D_\Sigma^2 u(x) - \sff_y + \; tg_\Sigma(x)) = \Delta_\Sigma u(x) + nt \quad \text{for } (x,y,t) \in V_0.
\end{equation}
We also define a map $\Phi : T^\perp\Sigma \times \mathbb R \longrightarrow \mathbb R^{n+m+1}$ by
$$\Phi(x,y,t) =\nabla^\Sigma u(x) + y + tx \quad \text{for } x \in \Sigma,\; y \in T_x^\perp\Sigma,\; t \in \mathbb R.$$

\begin{lemma}\label{Phi(V_r) in logSobolev}
We have $\{\xi \in \mathbb R^{n+m+1}: r \leq |\xi| < 1 \} = \Phi(V_r)$ for every $0 \leq r < 1$.
\end{lemma}
\begin{proof}
It is trivial from the definition of $V_r$ that $\Phi(V_r) \subseteq \mathbb B^{n+m+1}(1) \backslash \mathbb B^{n+m+1}(r)$. To see the reverse inclusion, let us pick a point $\xi$ in $\mathbb R^{n+m+1}$ satisfying $r \leq |\xi| < 1$. For every $x\in \Sigma$, we denote $\xi^{\text{tan}}_x$ the image of $\xi$ under the orthogonal projection onto $T_x\Sigma$. In other words,
$$\langle \xi^{\text{tan}}_x, V \rangle = \langle \xi , V \rangle \quad \text{for every } V \in T_x\Sigma.$$
We consider the function
$$v(x) \coloneqq u(x) - \langle \xi, x \rangle \quad \text{for } x \in \Sigma.$$
Since $\Sigma$ is closed, the function $v$ attains its global minimum at some point $x_0 \in \Sigma$. Then
$$0 = \nabla^\Sigma v(x_0) = \nabla^\Sigma u(x_0) - \xi^{\text{tan}}_{x_0}.$$
Hence, by noting that $x \perp y$ for every $y \in T_x^\perp \Sigma$, we conclude $\xi = \nabla^\Sigma u(x_0) + y_0 + t_0 x_0$ where $t_0 = \langle \xi, x_0 \rangle$ and $y_0$ is some vector in $T_{x_0}^\perp \Sigma$. In particular, since $\nabla^\Sigma u(x_0)$, $x_0$, and $y_0$ are perpendicular to each other, it follows that $|\nabla^\Sigma u(x_0)|^2 + |y_0|^2 + t_0^2 = |\xi|^2$ which is in $[r^2,1)$. Moreover, we also have
$$0 \leq D_\Sigma^2 v(x_0) = D_\Sigma^2 u(x_0) - \sff_{y_0} + \;  t_0 g_\Sigma(x_0).$$
Therefore, $(x_0, y_0, t_0) \in V_r$ and $\xi = \Phi(x_0, y_0, t_0)$. Since $\xi$ is arbitrarily chosen, this finishes the proof of the Lemma.
\end{proof}

\begin{lemma}\label{det(D Phi) in logSobolev}
For every $(x,y,t) \in V_0$, we have
\begin{equation}\label{eq:bound of t in logSoblev}
f(x)^{\frac{1}{n+1}} - \sqrt{1 - |\nabla^\Sigma u(x)|^2} + t \geq 0.
\end{equation}
Moreover, for every $(x,y,t) \in V_0$, the Jacobian determinant of $\Phi$ satisfies
\begin{equation}\label{eq:det(D Phi) in logSobolev}
|\det D\Phi(x,y,t)| \leq \left( f(x)^{\frac{1}{n+1}} - \sqrt{1 - |\nabla^\Sigma u(x)|^2} + t \right)^n.
\end{equation}
\end{lemma}
\begin{proof}
To start the proof of this Lemma, we fix an arbitrary point $(x,y,t) \in V_0$. We first show (\ref{eq:bound of t in logSoblev}). By the equation of $u$, we have
\begin{align*}
\Delta_\Sigma u(x) &= \frac{n}{n+1} \log f(x) - \frac{1}{2n}\frac{|\nabla^\Sigma f|^2}{f^2} - \left\langle  \frac{\nabla^\Sigma f}{f}, \nabla^\Sigma u \right\rangle.\\
& = \frac{n}{n+1} \log f(x) + \frac{n}{2} |\nabla^\Sigma u|^2 - \frac{1}{2}\left| \frac{1}{\sqrt{n}}\frac{\nabla^\Sigma f}{f} + \sqrt{n} \nabla^\Sigma u \right|^2\\
& \leq \frac{n}{n+1} \log f(x) + \frac{n}{2} |\nabla^\Sigma u|^2.
\end{align*}
Using the elementary inequalities
$$\log \lambda \leq \lambda - 1 \text{\; for } \lambda > 0 \quad \text{and} \quad \sqrt{1 - \theta} \leq 1 -\frac{\theta}{2} \text{\; for } 0 \leq \theta \leq 1,$$
we get
\begin{equation}\label{eq:my logSobolev elem ineq}
\frac{1}{n+1} \log f(x) \leq f(x)^{\frac{1}{n+1}} - 1 \quad \text{and} \quad \sqrt{1 - |\nabla^\Sigma u|^2} \leq 1 -\frac{|\nabla^\Sigma u|^2}{2}.
\end{equation}
Hence, we continue the chain of estimates for $\Delta_\Sigma u$ by
\begin{equation}\label{eq:Laplacian u in logSobolev}
\Delta_\Sigma u(x) \leq n\left( f(x)^{\frac{1}{n+1}} - 1 + \frac{|\nabla^\Sigma u|^2}{2}\right) \leq n\left(f(x)^{\frac{1}{n+1}} - \sqrt{1 - |\nabla^\Sigma u|^2}\right).
\end{equation}
Consequently, it follows from (\ref{eq:lower bound of Laplacian u + nt}) and (\ref{eq:Laplacian u in logSobolev}) that
\begin{equation}\label{eq:bound of t in logSoblev, proof}
0 \leq \Delta_\Sigma u(x) + nt \leq n\left(f(x)^{\frac{1}{n+1}} - \sqrt{1 - |\nabla^\Sigma u|^2} + t\right).
\end{equation}
This verifies (\ref{eq:bound of t in logSoblev}).

To show (\ref{eq:det(D Phi) in logSobolev}), we choose a local orthonomal frame $\{e_1, \dots, e_n, \nu_1, \dots, \nu_m\}$ at $(x,y)$ in the total space of the normal bundle $T^\perp \Sigma$ such that $e_i \in T_{x}\Sigma$ and $\nu_\alpha \in T^\perp_{x}\Sigma$. With respect to the frame $\{e_1, \dots, e_n, \nu_1, \dots, \nu_m, x\}$, at the point $(x,y,t)$, we compute
\begin{align*}
\langle \overline{\nabla}_{e_i} \Phi, e_j \rangle &= \langle \overline{\nabla}_{e_i} \nabla^\Sigma u, e_j \rangle + \langle \overline{\nabla}_{e_i} y, e_j \rangle + t\langle \overline{\nabla}_{e_i} x, e_j \rangle\\
&= (D^2_\Sigma u)(e_i, e_j) - \sff_y(e_i, e_j) + t\delta_{ij}.
\end{align*}
Moreover, since
$$\overline{\nabla}_{\nu_\alpha} \Phi = \overline{\nabla}_{\nu_\alpha}y = \nu_\alpha \quad \text{and} \quad \frac{\partial}{\partial t} \Phi = x,$$
we infer that
$$\langle \overline{\nabla}_{\nu_\alpha} \Phi, e_j \rangle =  0, \quad \langle \overline{\nabla}_{\nu_\alpha} \Phi, \nu_\beta\rangle = \delta_{\alpha\beta}, \quad \langle \overline{\nabla}_{\nu_\alpha} \Phi, x \rangle = 0$$
and
$$\left\langle \frac{\partial}{\partial t} \Phi, \, e_j \right\rangle = 0, \quad \left\langle \frac{\partial}{\partial t} \Phi, \,\nu_\beta \right\rangle = 0, \quad \left\langle \frac{\partial}{\partial t} \Phi, \, x \right\rangle = 1.$$
Therefore, the matrix representing $D\Phi$ has the form
\begin{equation}\label{eq:matrix of DPhi in logSobolev}
D\Phi(x,y,t) = \left(\begin{array}{@{}c | c@{}}
D_\Sigma^2 u(x) - \sff_y + \; tI_n& \quad * \quad \\
\hline
0 & I_{m+1}
\end{array}\right).
\end{equation}
Consequently, since $D_\Sigma^2 u(x) - \sff_y + \; tI_n \geq 0$ as $(x,y,t) \in V$, it follows from (\ref{eq:bound of t in logSoblev, proof}), (\ref{eq:matrix of DPhi in logSobolev}), and the arithmetic-geometric mean inequality that
\begin{align}
|\det D\Phi(x,y,t)| &= \det(D_\Sigma^2 u(x) - \sff_y + \;  tI_n) \nonumber \\
&\leq \left( \frac{\Delta_\Sigma u(x) + nt}{n} \right)^n \nonumber \\
&\leq \left( f(x)^{\frac{1}{n+1}} - \sqrt{1 - |\nabla^\Sigma u(x)|^2} + t \right)^n.\label{eq:det(DPhi) in logSobolev less than (Laplacian u/n + t)^n}
\end{align}
This verifies (\ref{eq:det(D Phi) in logSobolev}) and finishes the proof of the Lemma.
\end{proof}

By (\ref{eq:bound of t in logSoblev}) and the fact that $\Phi(V_0) = \mathbb B^{n+m+1}$ from Lemma \ref{Phi(V_r) in logSobolev}, for any $(x,y,t) \in V_0$, the range of $t$ is given by
$$\sqrt{1 - |\nabla^\Sigma u(x)|^2} - f(x)^{\frac{1}{n+1}} \leq t \leq \sqrt{1 - |\nabla^\Sigma u(x)|^2}.$$
In the next step, we follow the argument in \cite{Brendle3}. We begin by fixing any $0 < r < 1$. Then, we combine (\ref{det(D Phi) in logSobolev}), Lemma \ref{Phi(V_r) in logSobolev}, and Fubini's Theorem to obtain
\begin{align*}
|\mathbb B^{n+m+1}|&(1 - r^{n+m+1})\\
&= |\{\xi \in \mathbb R^{n+m+1}: r \leq |\xi| < 1 \}|\\
&\leq \int_\Omega \int_{\sqrt{1 - |\nabla^\Sigma u|^2} - f^{\frac{1}{n+1}}}^{\sqrt{1 - |\nabla^\Sigma u|^2}} \int_{\{y \in T_x^\perp \Sigma:\; r^2 \leq |\nabla^\Sigma u|^2 + t^2 + |y|^2 < 1\}}\\
&\qquad \qquad \qquad \qquad \qquad \qquad |\det D\Phi(x,y,t)|\, 1_{V_r}(x,y,t)\, dy\, dt\, d\vol_\Sigma(x)\\
& \leq \int_\Omega \int_{\sqrt{1 - |\nabla^\Sigma u|^2} - f^{\frac{1}{n+1}}}^{\sqrt{1 - |\nabla^\Sigma u|^2}} \int_{\{y \in T_x^\perp \Sigma:\; r^2 \leq |\nabla^\Sigma u|^2 + t^2 + |y|^2 < 1\}}\\
& \qquad \qquad \left( f(x)^{\frac{1}{n+1}} - \sqrt{1 - |\nabla^\Sigma u(x)|^2} + t \right)^n 1_{V_r}(x,y,t)\, dy\, dt\, d\vol_\Sigma(x).\\
& \leq \int_\Omega \int_{\sqrt{1 - |\nabla^\Sigma u|^2} - f^{\frac{1}{n+1}}}^{\sqrt{1 - |\nabla^\Sigma u|^2}} \int_{\{y \in T_x^\perp \Sigma:\; r^2 \leq |\nabla^\Sigma u|^2 + t^2 + |y|^2 < 1\}}\\
& \qquad \qquad \qquad \qquad \left( f(x)^{\frac{1}{n+1}} - \sqrt{1 - |\nabla^\Sigma u(x)|^2} + t \right)^n dy\, dt\, d\vol_\Sigma(x).
\end{align*}
Moreover, using the elementary inequality $b^\frac{m}{2} - a^\frac{m}{2} \leq \frac{m}{2}(b-a)$ for $0 \leq a \leq b < 1$ and $m \geq 2$, for every $x \in \Omega$, we have
\begin{align*}
|\{y \in &T_x^\perp \Sigma : r^2 \leq |\nabla^\Sigma u|^2 + t^2 + |y|^2 < 1\}|\\
& \leq |\mathbb B^m| \left((1 - |\nabla^\Sigma u|^2 - t^2)^\frac{m}{2}_+ - (r^2 - |\nabla^\Sigma u|^2 - t^2)^\frac{m}{2}_+ \right)\\
& \leq \frac{m}{2} |\mathbb B^m|(1-r^2).
\end{align*}
We therefore continue the previous chain of estimates by
\begin{align*}
|&\mathbb B^{n+m+1}|(1 - r^{n+m+1})\\
& \leq \frac{m}{2} |\mathbb B^m|(1-r^2)\int_\Omega \int_{\sqrt{1 - |\nabla^\Sigma u|^2} - f^{\frac{1}{n+1}}}^{\sqrt{1 - |\nabla^\Sigma u|^2}} \left( f^{\frac{1}{n+1}} - \sqrt{1 - |\nabla^\Sigma u|^2} + t \right)^n dt d\vol_\Sigma(x)\\
& \leq \frac{m}{2(n+1)} |\mathbb B^m|(1-r^2) \int_\Omega f(x)\; d\vol_\Sigma(x).
\end{align*}
Dividing the inequality
$$|\mathbb B^{n+m+1}|(1 - r^{n+m+1}) \leq \frac{m}{2(n+1)} |\mathbb B^m|(1-r^2) \int_\Omega f(x)\; d\vol_\Sigma(x)$$
by $(1-r)$, then sending $r \to 1$, we arrive at
\begin{equation}\label{eq:final equation my logSobolev}
|\mathbb S^{n+m}| \leq \frac{|\mathbb S^{m-1}|}{n+1} \int_\Omega f(x)\; d\vol_\Sigma(x) \leq \frac{|\mathbb S^{m-1}|}{n+1} \int_\Sigma f(x)\; d\vol_\Sigma(x).
\end{equation}
The last inequality coincides with (\ref{eq:goal of proof my logSobolev}). This finishes the proof of the Theorem when $\Sigma$ is connected.

Now, we suppose that $\Sigma$ is disconnected, Then, we apply (\ref{eq:my logSobolev m = 2}) to each individual of connected components of $\Sigma$, sum over them, and use the elementary inequality
$$a \log a + b\log b < (a+b)\log(a+b) \quad \text{for all } a, b > 0$$
to finish the proof in this case. The proof of Theorem \ref{my logSobolev} is complete.

\subsection{Proof of Theorem \ref{my logSoblev m = 2}}
Let $\Sigma$ be a closed $n$-dimensional minimal submanifold of $\mathbb S^{n+2}$ and $f$ be a positive smooth function on $\Sigma$. As pointed out in the introduction, inequality (\ref{eq:my logSobolev m = 2}) is a direct consequence of Theorem \ref{my logSobolev}. We only need to verify the rigidity statement. To get this done, we assume further that $f$ satisfies
\begin{equation}\label{eq:logSobolev rigidity}
\int_\Sigma f\left(\log f + \log(|\mathbb S^n|) \right) - \left(\int_\Sigma f \right) \log \left(\int_\Sigma f \right) = \frac{n+1}{2n^2} \int_\Sigma \frac{|\nabla^\Sigma f|^2}{f}.
\end{equation}
By the last paragraph in the proof of Theorem \ref{my logSobolev}, we conclude that $\Sigma$ is connected. After multiplying $f$ by a (unique) positive constant, we may assume without loss of generality that
\begin{equation}\label{eq:scaling in logSobolev rigidity}
\frac{n}{n+1} \int_\Sigma f(x) \log(f(x)) \; d\vol_\Sigma(x) = \frac{1}{2n} \int_\Sigma \frac{|\nabla^\Sigma f(x)|^2}{f(x)}\; d\vol_\Sigma(x).
\end{equation}
By (\ref{eq:logSobolev rigidity}), this means
\begin{equation}\label{eq:logSobolev rigidity goal}
|\mathbb S^n| = \int_\Sigma f(x) \, d\vol_\Sigma(x).
\end{equation}
Since $f > 0$ and $\Sigma$ is connected, by (\ref{eq:scaling in logSobolev rigidity}), there exists (uniquely modulo a constant) a smooth solution $u$ to the equation
$$\div(f\,\nabla^\Sigma u) = \frac{n}{n+1} f \log f - \frac{1}{2n} \frac{|\nabla^\Sigma f|^2}{f}.$$
We define the sets $\Omega$, $V_r$, and the map $\Phi$ as in the proof of Theorem \ref{my logSobolev}. It follows from (\ref{eq:final equation my logSobolev}) that $\Omega$ is dense in $\Sigma$.
\begin{lemma}\label{det(DPhi) in logSobolev rigidity}
For every $(x,y,t) \in V_0$, we have
$$|\det D\Phi(x,y,t)| = \left( f(x)^{\frac{1}{n+1}} - \sqrt{1 - |\nabla^\Sigma u(x)|^2} + t \right)^n.$$
\end{lemma}
\begin{proof}
Suppose that at some point $(x_0, y_0, t_0) \in V_0$, it holds
$$|\det D\Phi(x_0, y_0 ,t_0)| < \left( f(x_0)^{\frac{1}{n+1}} - \sqrt{1 - |\nabla^\Sigma u(x_0)|^2} + t \right)^n.$$
Then, by continuity and Lemma \ref{det(D Phi) in logSobolev}, there exists an open neighborhood $U$ of $(x_0, y_0, t_0)$ in $T^\perp\Sigma \times \mathbb R$ and a number $0 < \varepsilon < 1$ satisfying
$$|\det D\Phi(x,y,t)| < (1-\varepsilon) \left( f(x)^{\frac{1}{n+1}} - \sqrt{1 - |\nabla^\Sigma u(x)|^2} + t \right)^n \text{ in } U \cap V_0.$$
Let $U_{x_0}$ be image of $U$ under the projection onto $\Sigma$. In other words,
$$U_{x_0} \coloneqq \{x \in \Sigma\, \text{ such that } (x,y,t) \in  U \text{ for some } y \in T_x^\perp \Sigma,\,  t \in \mathbb R\}.$$
Clearly, $U_{x_0}$ contains an open neighborhood of $x_0$ in $\Sigma$. In particular, $|U_{x_0}| > 0$. We continue by fixing any $0 < r < r_0$. By Lemma \ref{det(D Phi) in logSobolev} and the argument in the proof of Theorem \ref{my logSobolev}, we have
\begin{align*}
|&\mathbb B^{n+3}|(1 - r^{n+3})\\
&= |\mathbb B^{n+3}(1) \backslash \mathbb B^{n+3}(r)|\\
& \leq \int_\Omega \int_{\sqrt{1 - |\nabla^\Sigma u|^2} - f^{\frac{1}{n+1}}}^{\sqrt{1 - |\nabla^\Sigma u|^2}} \int_{\{y \in T_x^\perp \Sigma:\; r^2 \leq |\nabla^\Sigma u|^2 + t^2 + |y|^2 < 1\}}\\
& \qquad (1 - \varepsilon 1_U(x,y,t))\left( f(x)^{\frac{1}{n+1}} - \sqrt{1 - |\nabla^\Sigma u(x)|^2} + t \right)^n 1_{V_r}(x,y,t)\, dy\, dt\, d\vol_\Sigma(x)\\
& \leq \int_\Omega \int_{\sqrt{1 - |\nabla^\Sigma u|^2} - f^{\frac{1}{n+1}}}^{\sqrt{1 - |\nabla^\Sigma u|^2}} \int_{\{y \in T_x^\perp \Sigma:\; r^2 \leq |\nabla^\Sigma u|^2 + t^2 + |y|^2 < 1\}}\\
& \qquad \qquad \quad \quad (1 - \varepsilon 1_U(x,y,t))\left( f(x)^{\frac{1}{n+1}} - \sqrt{1 - |\nabla^\Sigma u(x)|^2} + t \right)^n dy\, dt\, d\vol_\Sigma(x)\\
& \leq |\mathbb B^2|(1-r^2)\int_\Omega \int_{\sqrt{1 - |\nabla^\Sigma u|^2} - f^{\frac{1}{n+1}}}^{\sqrt{1 - |\nabla^\Sigma u|^2}}(1 - \varepsilon 1_{U_{x_0}}) \left( f^{\frac{1}{n+1}} - \sqrt{1 - |\nabla^\Sigma u|^2} + t \right)^n\\
&\qquad \qquad \qquad \qquad\qquad \qquad\qquad\qquad\qquad\qquad\qquad\qquad \qquad\qquad \qquad dt\, d\vol_\Sigma(x)\\
& \leq \frac{|\mathbb B^2|(1-r^2)}{n+1} \int_\Omega (1 - \varepsilon 1_{U_{x_0}}(x)) f(x)\; d\vol_\Sigma(x).
\end{align*}
Dividing the above inequalities by $(1-r)$, then passing to limit as $r \to 1$ and recalling the identity $(n+3)|\mathbb B^{n+3}| = 2|\mathbb B^2||\mathbb B^{n+1}| = 2|\mathbb B^2||\mathbb S^n|/(n+1)$, we obtain
\begin{align*}
|\mathbb S^n| \leq \int_\Omega (1 - \varepsilon 1_{U_{x_0}}(x)) f(x)\, d \vol_\Sigma(x) &\leq \int_\Sigma (1 - \varepsilon 1_{U_{x_0}}(x)) f(x)\, d \vol_\Sigma(x)\\
& < \int_\Sigma f(x)\, d \vol_\Sigma(x).
\end{align*}
This contradicts with (\ref{eq:logSobolev rigidity goal}). The proof of the Lemma is complete.
\end{proof}

Combining Lemma \ref{det(D Phi) in logSobolev} and Lemma \ref{det(DPhi) in logSobolev rigidity}, it follows that (\ref{eq:det(DPhi) in logSobolev less than (Laplacian u/n + t)^n}) and (\ref{eq:my logSobolev elem ineq}) must be equalities on $V_0$. The equalities in (\ref{eq:my logSobolev elem ineq}) implies that $f(x) = 1$ and $\nabla u(x) = 0$ for $(x,y,t) \in V_0$. Since $\Omega$ is the projection of $V_0$ to $\Sigma$ and dense in $\Sigma$, we conclude that $f = 1$ and $u$ is a constant on $\Sigma$. Consequently, for every $x \in \Sigma$ and sufficiently small $y \in T_x^\perp \Sigma$, the triplet $(x, y, \frac{1}{2})$ belongs to $V_0$. Moreover, the equality in (\ref{eq:det(DPhi) in logSobolev less than (Laplacian u/n + t)^n}) implies that $\sff_y$ is a multiple of $I_ n$ for every such $y$. As a minimal submanifold with such property, $\Sigma$ is totally geodesic. The proof of Theorem \ref{my logSoblev m = 2} is complete.

\section{Proof of Theorem \ref{my improved Serre}}\label{Proof of my improved Serre}
In this Section, we prove Theorem \ref{my improved Serre}. We divide the proof into two parts. In the first part, we prove the inequality (\ref{eq:my improved Serre}), and in the second part, we prove the rigidity statement.

\subsection{Proof of the inequality (\ref{eq:my improved Serre})}\label{subsec: Proof of the inequality (my improved Serre)}
Let $n \in \mathbb N$ and $\Omega$ be a smooth bounded domain in $\mathbb R^n$. Assume that $A$ is a smooth uniformly positive symmetric matrix-valued function on $\overline{\Omega}$. After multiplying $A$ by a (unique) positive constant, we may assume further that
\begin{equation}\label{eq:scaling in serre}
n \int_\Omega (\det A(x))^{\frac{1}{n-1}} dx = \int_{\partial \Omega} |A(x) \nu| d\sigma(x) + \int_\Omega |\Div A(x)| dx.
\end{equation}
Then we need to show
\begin{equation}\label{eq:mytheoremSerre goal}
| \mathbb B^n| \leq \int_{\Omega}  (\det A(x))^{\frac{1}{n-1}} dx.
\end{equation}
Since $A$ is positive and symmetric, by (\ref{eq:scaling in serre}), there exists (uniquely modulo a constant) a solution $u$ of the equations
\begin{equation*}
\begin{cases*}
\div(A(x) \nabla u) = n (\det A(x))^{\frac{1}{n-1}} - |\Div A(x)| \quad \text{in } \Omega\\
\langle A(x) \nabla u , \nu \rangle = |A(x)\nu| \qquad	\text{on } \partial \Omega.
\end{cases*}
\end{equation*}
 Since $n (\det A)^{1/(n-1)} - |\Div A|$ is Lipschitz continuous in $\overline{\Omega}$ and $A\nu$ is nowhere tangential on $\partial \Omega$, by the standard theory for elliptic equations, $u$ is in the class $C^{2,\alpha}(\overline{\Omega})$ for every $0 < \alpha < 1$ (see \cite{GilbargTrudingerBook}, Section 6.7). Let $V$ denote the set
$$V \coloneqq \{ x \in \Omega\, \text{ such that } |\nabla u(x)| < 1 \text{ and } D^2 u(x) \geq 0 \}.$$

\begin{lemma}\label{nabla(u) in my improved Serre}
We have $\mathbb B^n = \nabla u(V)$.
\end{lemma}

\begin{proof}
Clearly $\nabla u(V) \subseteq \mathbb B^n$. To see the reverse inclusion, we fix an arbitrary vector $\xi \in \mathbb B^n$. Then we consider the function
$$v(x) \coloneqq u(x) - \langle \xi, x \rangle \quad \text{for } x \in \overline{\Omega}.$$
Let $x_0$ be a minimum point of $v$. If $x_0 \in \partial \Omega$, then, since $A$ is positive symmetric on $\overline{\Omega}$, and $|\xi| < 1$, we would have
\begin{align*}
0 \geq \langle \nabla v(x_0), A(x_0)\nu \rangle &= \langle \nabla u(x_0), A(x_0)\nu \rangle - \langle \xi, A(x_0) \nu \rangle\\
&= |A(x_0)\nu| - \langle \xi, A(x_0) \nu \rangle\\
&> 0.
\end{align*}
Therefore, $x_0 \in \Omega$. Consequently, by the definition of $v$, we obtain
$$ \xi = \nabla u(x_0) \quad \text{and} \quad D^2u(x_0) \geq 0.$$
The proof of the Lemma is complete.
\end{proof}

\begin{lemma}\label{det(D^2u) in my improved Serre}
For every $x \in V$, we have
$$\det(D^2 u(x)) \leq  (\det A(x))^{\frac{1}{n-1}}.$$
\end{lemma}

\begin{proof}
We begin the proof by fixing a point $x \in V$. Note that
$$\div(A(x)\nabla u(x) = \tr(A(x)D^2u(x)) + \langle \Div A(x), \nabla u(x) \rangle.$$
Hence, by the equation of $u$, it follows that
\begin{align}
\tr(A(x)D^2u(x)) &= \div(A(x)\nabla u(x))  - \langle \Div A(x), \nabla u(x) \rangle \nonumber \\
&= n (\det A(x))^{\frac{1}{n-1}} - |\Div A(x)| - \langle \Div A(x), \nabla u(x) \rangle \nonumber \\
&\leq n (\det A(x))^{\frac{1}{n-1}}, \label{eq:tr(A D^2u) in my improved Serre}
\end{align}
where we have used the fact that $|\nabla u(x)| < 1$. Since $D^2 u(x) \geq 0$, we apply Lemma \ref{det(AB) < tr(AB)} to the pair $A(x)$, $D^2u(x)$ and use (\ref{eq:tr(A D^2u) in my improved Serre}) to obtain
\begin{align}
\det(D^2 u(x)) &=\frac{1}{\det A(x)} \det(A(x) D^2 u(x))\nonumber\\
&\leq\frac{1}{\det A(x)} \left(\frac{\tr (A(x) D^2 u(x))}{n} \right)^{n}\nonumber\\
&\leq (\det A(x))^{\frac{1}{n-1}}.\label{eq:det D^2 u in proof of my improved Serre}
\end{align}
This finishes the proof of the Lemma
\end{proof}

Applying Lemma \ref{nabla(u) in my improved Serre}, Lemma \ref{det(D^2u) in my improved Serre} and the area formula, we conclude that
\begin{align}
|\mathbb B^n| = |\nabla u(V)| \leq \int_V \det(D^2 u(x))dx &\leq \int_V (\det A(x))^{\frac{1}{n-1}}dx \nonumber \\
&\leq \int_\Omega (\det A(x))^{\frac{1}{n-1}}dx. \label{eq:last ineq in my improved Serre}
\end{align}
This coincides with (\ref{eq:mytheoremSerre goal}). The proof of (\ref{eq:my improved Serre}) is complete.

\subsection{Proof of the rigidity statement of Theorem \ref{my improved Serre}}

Let $n \in \mathbb N$ and $\Omega$ be a smooth bounded domain in $\mathbb R^n$. Assume that $A$ is a smooth uniformly positive symmetric matrix-valued function on $\overline{\Omega}$ such that the equality holds in (\ref{eq:my improved Serre}). After multiplying $A$ by a (unique) positive constant, we may assume further that
\begin{equation}\label{eq:scaling in serre rigidity}
n \int_\Omega (\det A(x))^{\frac{1}{n-1}} dx = \int_{\partial \Omega} |A(x) \nu| d\sigma(x) + \int_\Omega |\Div A(x)| dx.
\end{equation}
Equivalently, this means
\begin{equation}\label{eq:mytheoremSerre goal rigidity}
| \mathbb B^n| = \int_{\Omega}  (\det A(x))^{\frac{1}{n-1}} dx.
\end{equation}
Let $u$ and $V$ be the function and the set as in Section \ref{subsec: Proof of the inequality (my improved Serre)}. Then, it follows immediately from (\ref{eq:last ineq in my improved Serre}) that $V$ is a dense subset of $\Omega$. In particular, by continuity, we have that $D^2u \geq 0$ and $|\nabla u| \leq 1$ on $\Omega$.

\begin{lemma}\label{det(D^2u) in my improved Serre rigidity}
	For every $x \in V$, we have
	$$\det(D^2 u(x)) =  (\det A(x))^{\frac{1}{n-1}}.$$
\end{lemma}

\begin{proof}
In view of Lemma \ref{det(D^2u) in my improved Serre}, suppose, on the contrary, that there exists a point $x_0 \in V$ satisfying
$$\det(D^2 u(x_0)) < (\det A(x_0))^{\frac{1}{n-1}}.$$
By continuity, there is an open neighborhood $U$ of $x_0$ and a number $0 < \varepsilon < 1$ such that
$$\det(D^2 u(x)) < (1 - \varepsilon) (\det A(x)))^{\frac{1}{n-1}} \quad \text{for } x \in U \cap V.$$
Putting Lemma \ref{nabla(u) in my improved Serre}, Lemma \ref{det(D^2u) in my improved Serre} and this together, we deduce that
\begin{align*}
|\mathbb B^n| \leq |\nabla u(V)| &\leq \int_V \det(D^2 u(x))dx\\
&\leq \int_V (1 - \varepsilon 1_U(x))(\det A(x))^{\frac{1}{n-1}}dx\\
&\leq \int_\Omega (1 - \varepsilon 1_U(x))(\det A(x))^{\frac{1}{n-1}}dx \\
&< \int_\Omega (\det A(x))^{\frac{1}{n-1}}dx,
\end{align*}
which contradicts with (\ref{eq:mytheoremSerre goal rigidity}). The proof of the Lemma is complete.
\end{proof}

Since $V$ is dense in $\Omega$, Lemma \ref{det(D^2u) in my improved Serre rigidity} in particular gives
\begin{equation}\label{eq: D^2u is invertible in my Serre rigidity}
\det(D^2u(x)) = (\det A(x))^{\frac{1}{n-1}} > 0 \quad \text{for } x \in \overline{\Omega}.
\end{equation}
Therefore, recalling Lemma \ref{nabla(u) in my improved Serre} again, we deduce that
$$|\mathbb B^n| \leq |\nabla u(\Omega)| \leq \int_\Omega \det(D^2u(x)) dx = \int_\Omega (\det A(x))^{\frac{1}{n-1}} dx,$$
which in turn implies $|\mathbb B^n| =  |\nabla u(\Omega)|$ by (\ref{eq:mytheoremSerre goal rigidity}).

\begin{lemma}\label{A and u in my improved Serre rigidity}
The tensor $A$ is devergence-free in $\Omega$. Moreover, $u$ is smooth in $\Omega$ and $\nabla u(\Omega) = \mathbb B^n$.
\end{lemma}

\begin{proof}
Combining (\ref{eq: D^2u is invertible in my Serre rigidity}) and Lemma \ref{det(D^2u) in my improved Serre}, we infer that (\ref{eq:tr(A D^2u) in my improved Serre}) must be an equality. Since $|\nabla u| < 1$ in $V$, it follows that $\Div A = 0$ in $V$.  Because of density of $V$ in $\Omega$, we conclude that $A$ is divergence-free in $\Omega$. Hence, the equation of $u$ becomes
\begin{equation*}
\div(A(x) \nabla u) = n (\det A(x))^{\frac{1}{n-1}}\quad \text{in } \Omega.
\end{equation*}
Now, with the smooth right-hand side, by the standard theory for elliptic equations, $u$ is smooth in $\Omega$. Furthermore, by (\ref{eq: D^2u is invertible in my Serre rigidity}) and the inverse function theorem, $\nabla u(\Omega)$ is an open subset of $\mathbb R^n$. Since $\mathbb B^n \subseteq \nabla u(\Omega)$ and $|\mathbb B^n| =  |\nabla u(\Omega)|$, we conclude that $\mathbb B^n = \nabla u(\Omega)$. The proof of the Lemma is complete.
\end{proof}

\begin{lemma}\label{A = cof(D^2 u) in my Serre rigidity proof}
$A(x) = (\cof D^2 u(x))$, the cofactor matrix of $D^2u(x)$, for $x \in \Omega$.
\end{lemma}

\begin{proof}
According to (\ref{eq: D^2u is invertible in my Serre rigidity}) and Lemma \ref{det(D^2u) in my improved Serre}, we deduce that (\ref{eq:det D^2 u in proof of my improved Serre}) must be a chain of equalities. In particular, by Lemma \ref{det(AB) < tr(AB)} and density of $V$, in follows that $A(x) D^2 u(x) = \lambda(x)I_n$ for some positive smooth function $\lambda$ on $\Omega$ and for each $x \in \Omega$. Equivalently, since $D^2 u(x)$ is invertible, there is a positive smooth function $g$ on $\Omega$ satisfying
$$A(x) = g(x)(\cof D^2 u(x)) \qquad \text{for } x \in \Omega.$$
Recalling (\ref{eq: D^2u is invertible in my Serre rigidity}) once again, we obtain
$$\det A = g^n \det (\cof D^2u) = g^n (\det D^2 u)^{n-1} = g^n \det A.$$
Therefore, $g \equiv 1$. The proof of the Lemma is complete.
\end{proof}

Putting Lemma \ref{A and u in my improved Serre rigidity} and Lemma \ref{A = cof(D^2 u) in my Serre rigidity proof} together, we finish the proof of Theorem \ref{my improved Serre}.

\section{Proof of Theorem \ref{my quermass 2nd}}\label{sec:Proof of my quermass}

In this section, we use some properties of the Newton tensors (see e.g. \cite{Reilly73}). Let $n \in \mathbb N$ and $A$ be an $(n \times n)$-symmetric matrix and $1 \leq k \leq n-1$, we recall that the $k$-Newton tensor is defined by
$$T_k^{ij}(A) \coloneqq \frac{\partial \sigma_{k+1}(A)}{\partial A_{ij}}.$$
Since $\sigma_{k+1}$ is $(k+1)$-homogeneous, by the definition of $T_k$, the Euler's homogeneous function theorem yields
\begin{equation}\label{trace T_k times itself}
\sigma_{k+1}(A) = \frac{1}{k+1} \tr \left(T_k(A) \cdot A \right).
\end{equation}
The Newton tensors have an inductive formula, beginning with $T_0 = I_n$, and
\begin{equation}\label{T_k induction}
T_{k}(A) = \sigma_{k}(A)\,I_n - T_{k-1}(A) \cdot A.
\end{equation}
Taking trace of (\ref{T_k induction}) and using (\ref{trace T_k times itself}) gives
\begin{equation}\label{eq:tr(T_k)}
\tr T_{k}(A) = (n-k)\, \sigma_{k}(A).
\end{equation}
If we assume further that $A \in \Gamma_{k+1}$, then as mentioned in the introduction, $T_k(A)$ is positive definite. Moreover, Gårding's inequality for $\sigma_k$ has the following form:

\begin{proposition}[\cite{Garding}, Theorem 5]\label{Garding inequality for sigma_k}
	Let $n \in \mathbb N$ and $1 \leq k \leq n-1$. Suppose that $A$ and $B$ are square symmetric matrices of size $n$. If $A, B \in \Gamma_{k+1}$, then
	$$\tr(T_k(A)B) \geq (k+1) (\sigma_{k+1}(A))^\frac{k}{k+1} (\sigma_{k+1}(B))^\frac{1}{k+1}.$$
\end{proposition}

To prove Corollary \ref{my quermass result}, we show the following consequence of Gårding's inequality:

\begin{lemma}\label{det T_k > H_{k+1}}
Let $n \in \mathbb N$, $n \geq 2$, and $1 \leq k \leq n-2$. Suppose that $A$ is a square symmetric matrix of size $n$. If $A \in \Gamma_{k+1}$, then
\begin{equation}\label{eq:det T_k > H_{k+1}}
\det T_k(A) \geq {n-1 \choose k}^n H_{k+1}(A)^{\frac{nk}{k+1}}.
\end{equation}
The equality holds if and only if $A = \lambda I_n$ for some $\lambda > 0$.
\end{lemma}
Recall that $H_k(A)$ is the normalization of $\sigma_k(A)$, i.e, $H_k(A) \coloneqq \sigma_k(A) / {n \choose k}$. When $k = n-1$, (\ref{eq:det T_k > H_{k+1}}) becomes an equality.

\begin{proof}[Proof of Lemma \ref{det T_k > H_{k+1}}]
Since $A \in \Gamma_{k+1}$, its $k$-Newton tensor $T_k(A)$ is positive definite. Thus, by taking $B = \cof T_k(A)$ in Proposition $\ref{Garding inequality for sigma_k}$ and Maclaurin's inequality, it follows that
\begin{align}
n \det T_k(A) &= \tr(T_k(A) \cof T_k(A))\nonumber\\
&\geq (k+1)\, \sigma_{k+1}(A)^\frac{k}{k+1}\, \sigma_{k+1}(\cof T_k(A))^\frac{1}{k+1}\nonumber\\
&\geq (k+1)\, \sigma_{k+1}(A)^\frac{k}{k+1}\, {n \choose k+1}^\frac{1}{k+1} (\det \cof T_k(A))^\frac{1}{n}\label{eq:Maclurin in det T_k > H_{k+1} proof}\\
&= (k+1)\, \sigma_{k+1}(A)^\frac{k}{k+1}\, {n \choose k+1}^\frac{1}{k+1} (\det T_k(A))^\frac{n-1}{n}.\nonumber
\end{align}
Consequently, we obtain
\begin{align*}
\det T_k(A) \geq \left( \frac{k+1}{n} \right)^n {n \choose k+1}^\frac{n}{k+1}\sigma_{k+1}(A)^\frac{nk}{k+1} =  {n-1 \choose k}^n H_{k+1}(A)^{\frac{nk}{k+1}}.
\end{align*}
This finish the proof of (\ref{eq:det T_k > H_{k+1}}). If the equality in (\ref{eq:det T_k > H_{k+1}}) holds, then by (\ref{eq:Maclurin in det T_k > H_{k+1} proof}), we must have
$$\sigma_{k+1}(\cof T_k(A))^\frac{1}{k+1} = (\det \cof T_k(A))^\frac{1}{n}.$$
Hence, by Maclaurin's inequality, $\cof T_k(A) = \mu I_n$ for some $\mu > 0$. Equivalently, this means $T_k(A) = \eta I_n$ for some $\eta > 0$. Since $A \in \Gamma_{k+1}$, we therefore conclude that $A = \lambda I_n$ for some $\lambda > 0$. The proof of the Lemma is complete.
\end{proof}

Let $\Sigma$ be a closed hypersurface of $\mathbb R^{n+1}$ and recall the notations $\nu$, $\sff_\nu$ as in Section \ref{subsec:An inequality related to quermassintegrals}. If $\Sigma$ is $(k+1)$-convex, then by applying Lemma \ref{det T_k > H_{k+1}} with $A = \sff_\nu$, we obtain
$$\det T_k \geq {n-1 \choose k}^n H_{k+1}^{\frac{nk}{k+1}}.$$
As mentioned in Section \ref{subsec:An inequality related to quermassintegrals}, Corollary \ref{my quermass result} follows from this and Theorem \ref{my quermass 2nd}.\\

One important property of $T_k$ is the fact that it is divergence-free (see e.g. \cite{Reilly73}). More specifically, for every $x \in \Sigma$, if $\{e_1, e_2, \dots, e_n\}$ is a local orthonormal frame in a neighborhood of $x$ in $\Sigma$, $T_k$ satisfies
$$\nabla_{e_j} (T_k)_{ij} = 0 \quad \text{for each } 1 \leq i \leq n.$$

The proof of Theorem \ref{my quermass 2nd} is divided into two parts. In the first part, we prove the inequality (\ref{eq:my quermass 2nd}), and in the second part, we prove the rigidity statement.

\subsection{Proof of Theorem \ref{my quermass 2nd}: Inequality (\ref{eq:my quermass 2nd})}\label{Proof of Theorem {my quermass 2nd}: Inequality}
Let $n \in \mathbb N$, $n \geq 2$, $0 \leq k \leq n-2$, and $\Sigma$ be a closed $(k+1)$-convex hypersurface in $\mathbb R^{n+1}$. We assume for the moment that $\Sigma$ is connected. After rescaling $\Sigma$, we may assume without loss of generality that
\begin{equation}\label{eq:scaling in AF}
\int_\Sigma H_k \, d\vol_\Sigma = \int_\Sigma H_{k+1} \, d\vol_\Sigma.
\end{equation}
Then, we need to show
\begin{equation}\label{eq:my quermass 2nd goal}
|\mathbb S^n| \leq  {n-1 \choose k}^n \int_\Sigma \frac{H_k^{n+1}}{(\det T_k) H_{k+1}} \, d\vol_\Sigma.
\end{equation}
By the $(k+1)$-convexity assumption, the $k$-Newton tensor $T_k$ is positive definite. Thus, by (\ref{eq:scaling in AF}) and the assumption that $\Sigma$ is connected, there exists (uniquely modulo a constant) a smooth solution $u$ to the equation:
$$\div(T_k \nabla^\Sigma u) = (n-k){n \choose k} (H_k - H_{k+1}).$$
We define the map $\Phi$ by
\begin{align*}
\Phi :  \Sigma \times \mathbb R & \longrightarrow \mathbb R^{n+1}\\
(x,t) & \longmapsto \nabla^\Sigma u(x) + t \nu(x).
\end{align*}
We also define the sets $V \subset \Sigma \times \mathbb R$ and $\Omega \subset \Sigma$ by
\begin{align*}
V &\coloneqq \{(x,t) \in \Sigma \times \mathbb R :\, | \nabla^\Sigma u(x)|^2 + t^2 < 1 \text{ and } D_\Sigma^2u(x) - t \sff_\nu(x) \geq 0\},\\
\Omega &\coloneqq \{x \in \Sigma:\, (x, t) \in V \text{ for some } t \in \mathbb R\}.
\end{align*}

\begin{lemma}\label{Phi(V) in my quermass}
We have $\mathbb B^{n+1} = \Phi(V)$. 
\end{lemma} 
\begin{proof}
Clearly, $\Phi(V) \subseteq \mathbb B^{n+1}$. In order to show the reverse inclusion, we fix an arbitrary vector $\xi \in \mathbb B^{n+1}$. For every $x\in \Sigma$, we denote $\xi^{\text{tan}}_x$ the image of $\xi$ under the orthogonal projection onto $T_x\Sigma$. In other words,
$$\langle \xi^{\text{tan}}_x, V \rangle = \langle \xi , V \rangle \quad \text{for every } V \in T_x\Sigma.$$
Then we consider the function
$$v(x) \coloneqq u(x) - \langle \xi, x \rangle \quad \text{for } x \in \Sigma.$$
Since $\Sigma$ is closed, the function $v$ attains its minimum at some point $x_0 \in \Sigma$. Then
$$0 = \nabla^\Sigma v(x_0) = \nabla^\Sigma u(x_0) - \xi^{\text{tan}}.$$
Thus $\xi = \nabla^\Sigma u(x_0) + t_0\nu$ for some $t_0 \in \mathbb R$. Moreover, we also have
$$0 \leq D_\Sigma^2 v(x_0) = D_\Sigma^2 u(x_0) - t_0\sff_\nu(x_0),$$
and
$$|\nabla^\Sigma u(x_0)|^2 + t_0^2 = |\xi|^2 < 1.$$
This finishes the proof of the Lemma.
\end{proof}

\begin{lemma}\label{det(DPhi) in my equermass}
For every $(x,t) \in V$, we have $H_k - (1+t)H_{k+1} \geq 0$. Moreover, the Jacobian determinant of $\Phi$ satisfies
$$|\det D\Phi(x,t)| \leq \frac{1}{\det T_k} \left( {n-1 \choose k}(H_k - (1+t)H_{k+1}) \right)^n \text{ for } (x,t) \in V.$$
\end{lemma}
\begin{proof}
To see this, we fix arbtrarily $(x,t) \in V$. Then, we choose a local orthonomal frame $\{e_1, \dots, e_n\}$ at $x$. By the same computation as in the proof of Lemma \ref{det(D Phi) in logSobolev}, with respect to the frame $\{e_1, \dots, e_n, \nu \}$, the matix representing $D\Phi$ has the form
\begin{equation*}
D\Phi(x,t) = \left(\begin{array}{@{}c | c@{}}
D_\Sigma^2 u(x) - t\sff_\nu(x) & \; * \; \\
\hline
0 & 1
\end{array}\right).
\end{equation*}
Since $T_k$ is positive definite on $\Sigma$ and $D_\Sigma^2 u(x) - t\sff_\nu(x) \geq 0$, it follows that
\begin{equation*}
|\det D\Phi(x,t)| = \det(D_\Sigma^2 u(x) - t\sff_\nu(x)) = \frac{1}{\det T_k} \det(T_k(D_\Sigma^2 u - t\sff_\nu)).
\end{equation*}
On the other hand, using  (\ref{eq:tr(T_k)}), the equation of $u$, and the fact that $T_k$ is divergence-free, we have
\begin{align*}
0 \leq \tr(T_k(D_\Sigma^2 u - t\sff_\nu)) &= \tr(T_k\, D_\Sigma^2 u) - t \tr(T_k \sff_\nu)\\
&= \div(T_k \nabla^\Sigma u) - t \tr(T_k \sff_\nu)\\
&= (n-k){n \choose k}(H_k - (1+t)H_{k+1}).
\end{align*}
Consequently, putting these together and applying Lemma \ref{det(AB) < tr(AB)} to the pair $T_k$ and $D_\Sigma^2 u(x) - t\sff_\nu(x)$, we therefore obtain
\begin{align}
|\det D\Phi(x,t)| &= \frac{1}{\det T_k} \det(T_k(D_\Sigma^2 u - t\sff_\nu))\nonumber \\
&\leq \frac{1}{\det T_k} \left( \frac{1}{n} \tr(T_k(D_\Sigma^2 u - t\sff_\nu))\right)^n \label{eq:det(DPhi) in my quermass proof}\\
&= \frac{1}{\det T_k} \left( \frac{1}{n} (n-k){n \choose k}(H_k - (1+t)H_{k+1}) \right)^n. \nonumber
\end{align}
From here, using the indentiy $\frac{1}{n}(n-k){n \choose k} = {n-1 \choose k}$, we finish the proof of the Lemma.
\end{proof}

It follows from Lemma \ref{Phi(V) in my quermass} and Lemma \ref{det(DPhi) in my equermass} that for every $(x,t) \in V$, the range of $t$ satisfies $-1 \leq t \leq \frac{H_k}{H_{k+1}} - 1$. Moreover, we have
\begin{align}
|\mathbb B^{n+1}| &= |\Phi(V)|\nonumber\\
&\leq \int_\Omega \int_{-1}^{\frac{H_k}{H_{k+1}}-1}  |\det D\Phi(x,t)|\, 1_V(x,t)\, dt\, d\vol_\Sigma(x)\nonumber\\
&\leq \int_\Omega \int_{-1}^{\frac{H_k}{H_{k+1}}-1} \frac{1}{\det T_k} \left(  {n-1 \choose k}(H_k - (1+t)H_{k+1}) \right)^n 1_V(x,t)\, dt\, d\vol_\Sigma(x)\nonumber\\
&\leq \int_\Omega \int_{-1}^{\frac{H_k}{H_{k+1}}-1} \frac{1}{\det T_k} \left(  {n-1 \choose k}(H_k - (1+t)H_{k+1}) \right)^n dt\, d\vol_\Sigma\nonumber\\
&= \frac{1}{n+1} {n-1 \choose k}^n \int_\Omega \frac{H_k^{n+1}}{(\det T_k) H_{k+1}}\, d\vol_\Sigma\nonumber\\
&\leq \frac{1}{n+1} {n-1 \choose k}^n \int_\Sigma \frac{H_k^{n+1}}{(\det T_k) H_{k+1}}\, d\vol_\Sigma. \label{eq:last equation in my quermass 2nd ineq}
\end{align}
This coincides with (\ref{eq:my quermass 2nd goal}) and therefore proves (\ref{eq:my quermass 2nd}) when $\Sigma$ is connected.

Now, suppose that $\Sigma$ is disconnected. Let $\Sigma_1$ be one of its connected component. We apply (\ref{eq:my quermass 2nd}) to $\Sigma_1$ and rewrite the resulting inequality as
\begin{equation}\label{eq:my quermass 2nd connected components}
|\mathbb S^n|^\frac{1}{n+1} \left(\int_{\Sigma_1} H_k\right) \leq {n-1 \choose k}^\frac{n}{n+1} \left(\int_{\Sigma_1} H_{k+1}\right) \left(\int_{\Sigma_1} \frac{H_k^{n+1}}{(\det T_k)H_{k+1}}\right)^\frac{1}{n+1}.
\end{equation}
Then, we sum \ref{eq:my quermass 2nd connected components} over its connected components and use the elementary inequality
$$a_1 \, . \, b_1^\frac{1}{n+1} + a_2 \, . \, b_2^\frac{1}{n+1} < (a_1 + a_2) \, . \, (b_1 + b_2)^\frac{1}{n+1} \quad \text{for } a_1, a_2, b_1, b_2 > 0$$
to finish the proof of  (\ref{eq:my quermass 2nd}) in this case.

\subsection{Proof of Theorem \ref{my quermass 2nd}: Rigidity}

Let $n \in \mathbb N$, $n \geq 2$, $0 \leq k \leq n-2$ and $\Sigma$ be a closed $(k+1)$-convex hypersurface in $\mathbb R^{n+1}$ such that the equality holds in (\ref{eq:my quermass 2nd}). By the argument from the last paragraph of Section \ref{Proof of Theorem {my quermass 2nd}: Inequality}, we conclude that $\Sigma$ is connected. After rescaling $\Sigma$, we may assume without loss of generality that
\begin{equation}\label{eq:scaling in my quermass 2nd rigidity}
\int_\Sigma H_k \, d\vol_\Sigma = \int_\Sigma H_{k+1} \, d\vol_\Sigma.
\end{equation}
Equivalently, this means
\begin{equation}\label{eq:my quermass 2nd rigidity goal}
|\mathbb S^n| =  {n-1 \choose k}^n \int_\Sigma \frac{H_k^{n+1}}{(\det T_k) H_{k+1}} \, d\vol_\Sigma.
\end{equation}
We continue to define the function $u$, the sets $V$, $\Omega$ and the map $\Phi$ as in Section \ref{Proof of Theorem {my quermass 2nd}: Inequality}. Then, for (\ref{eq:my quermass 2nd rigidity goal}) to hold, (\ref{eq:last equation in my quermass 2nd ineq}) must be the chain of equalities. In particular, the last equality in (\ref{eq:last equation in my quermass 2nd ineq}) implies that $\Omega$ is dense in $\Sigma$. Moreover, if we define the set $\widetilde{\Sigma}$ by
$$\widetilde{\Sigma} \coloneqq \{(x,t):\, x \in \Sigma \, \text{ and } -1 \leq t \leq \frac{H_k}{H_{k+1}} - 1 \text{ at } x\},$$
then it follows from the fourth and fifth equalities in the chain  (\ref{eq:last equation in my quermass 2nd ineq}) that $V$ is dense in $\widetilde{\Sigma}$. By continuity, we conclude that $D_\Sigma^2u(x) - t \sff_\nu(x) \geq 0$ for every $(x,t) \in \widetilde{\Sigma}$. Consequently, by the second equality in the chain (\ref{eq:last equation in my quermass 2nd ineq}) and Lemma \ref{det(DPhi) in my equermass}, we infer that
$$|\det D\Phi(x,t)| = \frac{1}{\det T_k} \left( {n-1 \choose k}(H_k - (1+t)H_{k+1}) \right)^n \text{ for } (x,t) \in \widetilde{\Sigma}.$$
Thus, by (\ref{eq:det(DPhi) in my quermass proof}) and density of $V$ in $\widetilde{\Sigma}$, we must have that
$$\det(T_k(D_\Sigma^2 u - t\sff_\nu))  = \left( \frac{1}{n} \tr(T_k(D_\Sigma^2 u - t\sff_\nu))\right)^n \quad \text{at every } (x,t) \in \widetilde{\Sigma}.$$
Hence, by fixing a point $x\in \Sigma$ while letting $t$ vary and applying Lemma \ref{det(AB) < tr(AB)} to the pair $T_k$ and $D_\Sigma^2 u - t\sff_\nu$ at $x$, we derive $T_k\sff_\nu = \lambda I_n$ where $\lambda = (k+1)\sigma_{k+1}(\sff_\nu) > 0$. So $\sff_\nu = \lambda(T_k)^{-1}$ which is positive definite. Furthermore, using (\ref{T_k induction}), we obtain $T_{k+1} = \eta I_n$ for some $\eta > 0$. Since $k \leq n-2$ and $\sff_\nu \in \Gamma_n$, this implies that $x$ is an umbilical point. Because $x$ is arbitrarily chosen, we conclude that $\Sigma$ is a sphere. The proof of Proposition \ref{my quermass 2nd} is complete.

\section*{Acknowledgements}
This is part of the author's thesis. The author is grateful to his advisor, Professor YanYan Li, for extensive discussions and helpful suggestions.

\appendix
\section{An auxiliary lemma}\label{sec:Appendix}

 For readers' convenience, we recall a standard fact that have been used in the main parts of the paper.
 
 \begin{lemma}\label{det(AB) < tr(AB)}
For $n \in \mathbb N$, let $A$ and $B$ be square symmetric matrices of size $n$. Assume that $A$ is positive definite and $B$ is non-negative definite. Then
 \begin{equation}\label{eq:det(AB) < tr(AB)}
 \det(AB) \leq \left(\frac{\tr(AB)}{n}\right)^n.
 \end{equation}
 The equality holds if and only if $AB = \lambda I_n$ for some $\lambda \geq 0$, where $I_n$ is the identity matrix.
 \end{lemma}
See e.g. \cite{Serre} for the proof of the lemma and further discussions.

\end{document}